\documentclass[10pt,a4paper]{article}

\usepackage{latexsym,amsfonts,amsmath,amssymb,amsthm,mathrsfs,color}

\newtheorem{theorem}{Theorem}
\newtheorem{lemma}{Lemma}

\newtheorem{remark}{Remark}
\newtheorem{proposition}{Proposition}
\newtheorem{definition}{Definition}
\newtheorem{corollary}{Corollary}
\usepackage[dvipdfm,
linkbordercolor={1 0 0},
citebordercolor={0 1 0},
urlbordercolor={0 1 1}]{hyperref}

\newcommand{\rd}{\, \mathrm{d}}
\newcommand{\bszero}{\boldsymbol{0}}
\newcommand{\bsc}{\boldsymbol{c}}
\newcommand{\bsd}{\boldsymbol{d}}
\newcommand{\bsk}{\boldsymbol{k}}
\newcommand{\bsl}{\boldsymbol{l}}
\newcommand{\bsr}{\boldsymbol{r}}
\newcommand{\bsp}{\boldsymbol{p}}
\newcommand{\bsq}{\boldsymbol{q}}
\newcommand{\bsx}{\boldsymbol{x}}
\newcommand{\bsy}{\boldsymbol{y}}

\newcommand{\bsalpha}{\boldsymbol{\alpha}}

\newcommand{\FF}{\mathbb{F}}
\newcommand{\NN}{\mathbb{N}}
\newcommand{\RR}{\mathbb{R}}

\newcommand{\Dcal}{\mathcal{D}}

\newcommand{\Jcal}{\mathcal{J}}
\newcommand{\Scal}{\mathcal{S}}
\newcommand{\per}{\mathrm{per}}

\newcommand{\wal}{\mathrm{wal}}
\newcommand{\wor}{\mathrm{wor}}

\allowdisplaybreaks

\begin{document}

\title{Optimal order quadrature error bounds for infinite-dimensional higher order digital sequences}

\author{Takashi Goda\thanks{Graduate School of Engineering, The University of Tokyo, 7-3-1 Hongo, Bunkyo-ku, Tokyo 113-8656, Japan (\tt{goda@frcer.t.u-tokyo.ac.jp})},
Kosuke Suzuki\thanks{School of Mathematics and Statistics, The University of New South Wales, Sydney 2052, Australia. ({\tt kosuke.suzuki1@unsw.edu.au})},
Takehito Yoshiki\thanks{School of Mathematics and Statistics, The University of New South Wales, Sydney 2052, Australia. ({\tt takehito.yoshiki1@unsw.edu.au})}}

\date{\today}

\maketitle

\begin{abstract}
Quasi-Monte Carlo (QMC) quadrature rules using higher order digital nets and sequences have been shown to achieve the almost optimal rate of convergence of the worst-case error in Sobolev spaces of arbitrary fixed smoothness $\alpha\in \mathbb{N}$, $\alpha\geq 2$.
In a recent paper by the authors, it was proved that randomly-digitally-shifted order $2\alpha$ digital nets in prime base $b$ achieve the best possible rate of convergence of the root mean square worst-case error of order $N^{-\alpha}(\log N)^{(s-1)/2}$ for $N=b^m$, where $N$ and $s$ denote the number of points and the dimension, respectively, which implies the existence of an optimal order QMC rule.
More recently, the authors provided an explicit construction of such an optimal order QMC rule by using Chen-Skriganov's digital nets in conjunction with Dick's digit interlacing composition.
These results were for fixed number of points.
In this paper we give a more general result on an explicit construction of optimal order QMC rules for arbitrary fixed smoothness $\alpha\in \mathbb{N}$ including the endpoint case $\alpha=1$.
That is, we prove that the projection of any infinite-dimensional order $2\alpha +1$ digital sequence in prime base $b$ onto the first $s$ coordinates achieves the best possible rate of convergence of the worst-case error of order $N^{-\alpha}(\log N)^{(s-1)/2}$ for $N=b^m$.
The explicit construction presented in this paper is not only easy to implement but also extensible in both $N$ and $s$.
\end{abstract}

\section{Introduction and statement of the main result}
In this paper we study numerical integration of smooth functions defined over the $s$-dimensional unit cube.
Let $f\colon [0,1)^s\to \RR$ be an integrable function.
We denote the true integral of $f$ by
\begin{align*}
I(f) := \int_{[0,1)^s}f(\bsx)\rd \bsx.
\end{align*}
A quasi-Monte Carlo (QMC) rule is an equal-weight quadrature rule where the weights sum up to 1. That is, by using an $N$-element point set $P_{N}^{(s)}=\{\bsx_{n}^{(s)}\colon 0\leq n<N\}\subset [0,1)^s$, a QMC rule approximates $I(f)$ by
\begin{align*}
I(f; P_{N}^{(s)}) = \frac{1}{N}\sum_{n=0}^{N-1}f(\bsx_{n}^{(s)}).
\end{align*}
In case of an infinite sequence of points $\Scal^{(s)}=\{\bsx_n^{(s)}\colon n\geq 0\}\subset [0,1)^s$, we use the first $N$ elements of $\Scal^{(s)}$ as $P_{N}^{(s)}$.
Moreover, in case of an infinite-dimensional sequence of points $\Scal=\{\bsx_n\colon n\geq 0\}\subset [0,1)^{\NN}$, we use the first $N$ points of the projection of $\Scal$ on to the first $s$ coordinates as $P_{N}^{(s)}$.
An explicit construction of a good infinite-dimensional sequence of points is of considerable importance in practical applications since then the point set is extensible in both $N$ and $s$, i.e., both the number of points and the dimension can be increased while retaining the existing points \cite{Hic03}.

In order to measure the quality of a QMC point set or sequence for a class of integrands instead of a single integrand, we consider the so-called worst-case error for a normed function space which is defined as follows.
Let $V$ be a function space with norm $\|\cdot\|_V$.
The worst-case error of a point set $P_{N}^{(s)}$ in $V$, denoted by $e^{\wor}(V,P_{N}^{(s)})$, is the supremum of the quadrature error in the unit ball of $V$, i.e.,
\begin{align*}
e^{\wor}(V,P_{N}^{(s)}) := \sup_{\substack{f\in V\\ \|f\|_V \leq 1}}|I(f;P_{N}^{(s)})-I(f)|.
\end{align*}
Reproducing kernel Hilbert spaces of Sobolev type consisting of functions which have square integrable partial mixed derivatives up to order 1 in each variable have been frequently studied as an important example of a normed function space in the literature, see for instance \cite{Hic98} and \cite[Appendix~A]{NWbook}.
The worst-case error in such function spaces is often connected with geometric discrepancy of point sets, which measures how uniformly distributed the point set is.
This fact has motivated to construct low-discrepancy point sets and sequences.
Indeed there are many known explicit constructions of (possibly infinite-dimensional) low-discrepancy point sets and sequences including those of Halton \cite{Hal60}, Sobol' \cite{Sob67}, Faure \cite{Fau82}, Niederreiter \cite{Nie88}, Niederreiter and Xing \cite{NXbook}, Chen and Skriganov \cite{CS02} and Dick and Pillichshammer \cite{DP14} to list just a few.

Recently, function spaces with higher smoothness $\alpha\in \NN$, $\alpha \geq 2$, have attracted considerable attention in the area of uncertainty quantification, in particular partial differential equations with random coefficients, see for instance \cite{DKLNS14,DLS16}.
For such smooth functions it is possible to achieve higher order of convergence of the worst-case error by exploiting the smoothness of the functions.
In his seminal work \cite{Dic07,Dic08}, Dick studied the decay of the Walsh coefficients of smooth functions and introduced a digit interlacing composition for constructing the so-called (possibly infinite-dimensional) higher order digital nets and sequences which achieve the almost optimal rate of convergence $N^{-\alpha}(\log N)^{c(s,\alpha)}$ for $N=b^m$ and some $c(s,\alpha)>0$, see \cite[Theorem~5.4 \& Corollary~5.5]{Dic08} and also \cite[Theorem~22]{BD09}.
(We refer to Subsection~\ref{subsec:ho_digital_net} for the definition of higher order digital nets and sequences.)
Although this order of convergence is best possible up to some power of a $\log N$ factor, it has been unknown until recently whether the exponent $c(s,\alpha)$ can be improved to be optimal.

There has been some progress reported on the optimality of the exponent $c(s,\alpha)$ for higher order digital nets.
In \cite{HMOT15}, Hinrichs et al.\ considered periodic Sobolev spaces and periodic Nikol'skij-Besov spaces with (real-valued) dominating mixed smoothness up to 2, and obtained $c(s,\alpha)=(s-1)/2$ for order 2 digital nets, which is best possible.
In \cite{GSY1}, the authors of this paper considered a reproducing kernel Hilbert space $H_{\alpha,s}$ of Sobolev type consisting of non-periodic functions with smoothness $\alpha\in \NN$, $\alpha \geq 2$, and proved that randomly-digitally-shifted order $2\alpha$ digital nets can achieve the rate of convergence of the root mean square worst-case error of order $N^{-\alpha}(\log N)^{(s-1)/2}$.
(We refer to Subsection~\ref{subsec:sobolev} for the definition of $H_{\alpha,s}$.)
This result directly implies the existence of a QMC rule achieving $c(s,\alpha)=(s-1)/2$, which is best possible \cite[Proposition~1.3]{GSY1}.
More recently in \cite{GSY2}, the authors of this paper provided an explicit construction of such an optimal order QMC rule by using Chen-Skriganov's digital nets in conjunction with Dick's digit interlacing composition.

In this third paper of the authors on the optimality of higher order digital nets and sequences, we prove the following theorem:
\begin{theorem}\label{thm:main}
Let $b$ be a prime, and $s,\alpha\in \NN$.
Let $\Scal$ be an infinite-dimensional order $2\alpha+1$ digital sequence over the finite field $\FF_b$, and let $P_{N}^{(s)}$ be the set of the first $N$ points of the projection of $\Scal$ onto the first $s$ coordinates for $N\in \NN$.
Then for any $m\in \NN$, we have
\begin{align*}
e^{\wor}(H_{\alpha,s}; P_{b^m}^{(s)}) \leq C_{\alpha,b,s}\frac{m^{(s-1)/2}}{b^{\alpha m}},
\end{align*}
where $C_{\alpha,b,s}$ is positive and independent of $m$.
\end{theorem}
\noindent
Theorem~\ref{thm:main} claims that the projection of any infinite-dimensional order $2\alpha +1$ digital sequence onto the first $s$ coordinates achieves the best possible rate of convergence $N^{-\alpha}(\log N)^{(s-1)/2}$ of the worst-case error in $H_{\alpha,s}$, when $N$ is of the form $b^m$ for some $m\in \NN$ with a fixed prime base $b$.
Thus our result implies a new explicit construction of optimal order QMC rules for smooth integrands, which is not only easy to implement but also extensible in both $N$ and $s$.
As shown in \cite{HKKN12,Owe16}, in order for an extensible QMC rule to achieve a rate of convergence better than $N^{-1}$, the number of points must grow at least geometrically.
Therefore, our result is also best possible in this regard when $\alpha \geq 2$.

We would remark that our present result covers the endpoint case $\alpha=1$, which has not been included in the previous results \cite{GSY1,GSY2}.
Furthermore, since any order $2\alpha +1$ digital sequence is also an order $2\alpha' +1$ digital sequence for all $\alpha' < \alpha$ due to the so-called propagation rules \cite{BDP11}, it also achieves the best possible rate of convergence $N^{-\alpha'}(\log N)^{(s-1)/2}$ of the worst-case error in $H_{\alpha',s}$ as long as $\alpha' < \alpha$.
The main idea of our proof of Theorem~\ref{thm:main} is to exploit the decay and the sparsity of the Walsh coefficients of the reproducing kernel simultaneously in a different way from \cite{GSY2}.
A similar approach is used in \cite{DP14} to prove the best possible order of the $L_2$ discrepancy bound for higher order digital sequences over $\FF_2$.
Whether it is possible to lower the necessary order of digital sequences from $2\alpha+1$ while still achieving the best possible rate of convergence is an open question.

We note that the optimal order of convergence in the Sobolev space with smoothness $\alpha$ is also achieved by using the Frolov lattice rule in conjunction with a periodization technique \cite{Fro76,NUU15,Ull14,UU15}.
In fact, it has been shown that the Frolov lattice rule achieves the optimal order of convergence in Besov and Triebel-Lizorkin spaces with dominating mixed smoothness, and the optimality in the Sobolev space is obtained from the fact that the Sobolev space coincides with the Triebel-Lizorkin space under the special choice of the parameters.
Although the Frolov lattice rule is an equal-weight quadrature rule, the weights do not sum up to 1 in general, so that it is not a QMC rule.
Periodizing an integrand can be interpreted as having a modified quadrature rule with possibly unequal weights of the original integrand \cite{Hic02}, and is shown to cause a severe consequence especially in high dimensions even for the constant functions \cite{KSW07}.
Since our optimal order QMC rule does not require such a periodization, a numerical stability is kept independently of the dimension.
Nevertheless, a significant advantage of the Frolov lattice rule lies in its universality, that is, one quadrature rule satisfies the optimal order worst-case error bounds for many function spaces (such as Besov, Triebel-Lizorkin, and Sobolev) with different smoothness.
As mentioned above, our presented optimal order QMC rules also possess a kind of universality, which is though weaker than that of the Frolov lattice rule.
On the other hand, the Frolov lattice rule is not extensible in either $N$ or $s$ and also seems not easy to implement in high dimensions.
Finally we would point out that both of the quadrature rules have randomized and unbiased variants \cite{Dic11,KN16}.

\section{Preliminaries}\label{sec:pre}
Throughout this paper we shall use the following notations.
Let $\NN$ be the set of positive integers and $\NN_0=\NN\cup \{0\}$.
Let $b$ be a fixed prime, and $\FF_b$ the finite field with $b$ elements, which we identify with the set $\{0,1,\ldots,b-1\}$ equipped with addition and multiplication modulo $b$.
The operators $\oplus$ and $\ominus$ denote digitwise addition and subtraction modulo $b$, respectively, that is, for $k=\sum_{i=1}^{\infty}\kappa_ib^{i-1}\in \NN_0$ and $k'=\sum_{i=1}^{\infty}\kappa'_ib^{i-1}\in \NN_0$ with $\kappa_i,\kappa'_i\in \FF_b$, where all, except a finite number of $\kappa_i$ and $\kappa'_i$, are 0, we define
\begin{align*}
k\oplus k' := \sum_{i=1}^{\infty}\lambda_ib^{i-1}\quad \text{and}\quad k\ominus k' := \sum_{i=1}^{\infty}\lambda'_ib^{i-1},
\end{align*}
where $\lambda_i=\kappa_i+\kappa'_i \pmod b$ and $\lambda'_i=\kappa_i-\kappa'_i \pmod b$.
In case of vectors in $\NN_0^s$, the operators $\oplus$ and $\ominus$ are applied componentwise.

\subsection{Higher order digital nets and sequences}\label{subsec:ho_digital_net}

\subsubsection{Definitions}
We start with a general digital construction scheme of infinite-dimensional finite point sets in the unit cube due to Niederreiter \cite[p.~63]{Nbook}.
\begin{definition}\label{def:digital_net}
For $m,n\in \NN$, let $C_j\in \FF_b^{n\times m}$ for $j\in \NN$ be $n\times m$ matrices over $\FF_b$.
For an integer $0\le h<b^m$, we denote the $b$-adic expansion of $h$ by $h=\sum_{i=1}^{m}\eta_i b^{i-1}$.
For $j\in \NN$, let
\begin{align*}
 x_{h,j} = \frac{\xi_{1,h,j}}{b}+\frac{\xi_{2,h,j}}{b^2}+\cdots + \frac{\xi_{n,h,j}}{b^n} ,
\end{align*}
where $\xi_{1,h,j},\xi_{2,h,j},\ldots,\xi_{n,h,j}$ are given by
\begin{align*}
 (\xi_{1,h,j},\xi_{2,h,j},\ldots,\xi_{n,h,j})^{\top} = C_j (\eta_1,\eta_2,\ldots,\eta_m)^{\top}.
\end{align*}
Then the set $P_{b^m}=\{\bsx_0,\bsx_1,\ldots,\bsx_{b^m-1}\}$ with $\bsx_h=(x_{h,1},x_{h,2},\ldots)\in [0,1)^{\NN}$ is called an infinite-dimensional digital net over $\FF_b$ with generating matrices $C_j$.
Moreover, the projection of $P_{b^m}$ onto the first $s$ coordinates is called a digital net over $\FF_b$ with generating matrices $C_1,\ldots,C_s$.
\end{definition}
\noindent
This construction scheme can be extended to the case of an infinite-dimensional sequence of points \cite[p.~72]{Nbook}.
\begin{definition}\label{def:digital_seq}
Let $C_j\in \FF_b^{\NN\times \NN}$ for $j\in \NN$ be $\NN\times \NN$ matrices over $\FF_b$.
For each $C_j=(c_{k,l}^{(j)})_{k,l\in \NN}$ we assume that there exists a function $K: \NN\to \NN$ such that $c_{k,l}^{(j)}=0$ when $k>K(l)$.
For a non-negative integer $h$, we denote the $b$-adic expansion of $h$ by $h=\sum_{i=1}^{a}\eta_i b^{i-1}$ for some $a\in \NN$.
For $j\in \NN$, let
\begin{align*}
 x_{h,j} = \frac{\xi_{1,h,j}}{b}+\frac{\xi_{2,h,j}}{b^2}+\cdots ,
\end{align*}
where $\xi_{1,h,j},\xi_{2,h,j},\ldots$ are given by
\begin{align*}
 (\xi_{1,h,j},\xi_{2,h,j},\ldots)^{\top} = C_j (\eta_1,\eta_2,\ldots,\eta_a,0,0,\ldots)^{\top}.
\end{align*}
Then the sequence $\Scal=(\bsx_0,\bsx_1,\ldots)$ with $\bsx_h=(x_{h,1},x_{h,2},\ldots)\in [0,1)^{\NN}$ is called an infinite-dimensional digital sequence over $\FF_b$ with generating matrices $C_j$.
Moreover, the projection of $\Scal$ onto the first $s$ coordinates is called a digital sequence over $\FF_b$ with generating matrices $C_1,\ldots,C_s$.
\end{definition}
\noindent 
Note that the condition $c_{k,l}^{(j)}=0$ for all sufficiently large $k$ is a standard assumption to ensure that every number $x_{h,j}$ is $b$-adic rational, i.e., $x_{h,j}$ is written in a finite $b$-adic expansion, so that every point in $\Scal$ belongs to $[0,1)^{\NN}$, see for instance \cite{DP14,Nbook}.

Now we give the definitions of higher order digital nets and sequences, which are special cases of \cite[Definition~4.3]{Dic08} and \cite[Definition~4.8]{Dic08}, respectively.
\begin{definition}\label{def:ho_digital_net}
For $m,n,s,\alpha\in \NN$ with $n\geq \alpha m$, let $P_{b^m}^{(s)}$ be a digital net over $\FF_b$ with generating matrices $C_1,\ldots,C_s\in \FF_b^{n\times m}$.
For $1\leq i\leq n$ and $1\leq j\leq s$, we denote the $i$-th row of $C_j$ by $\bsc_{i,j}$.
Let $t$ be an integer with $0\leq t\leq \alpha m$ which satisfies the following condition:
For all $1\leq i(j,v_j)<\cdots < i(j,1)\leq n$ with
\begin{align*}
 \sum_{j=1}^{s} \sum_{l=1}^{\min(\alpha, v_j)}i(j,l)\leq \alpha m -t,
\end{align*}
the vectors $\bsc_{i(1,v_1),1},\ldots,\bsc_{i(1,1),1}, \ldots, \bsc_{i(s,v_s),s},\ldots,\bsc_{i(s,1),s}$ are linearly independent over $\FF_b$. Then we call $P_{b^m}^{(s)}$ an order $\alpha$ digital $(t,m,s)$-net over $\FF_b$.
\end{definition}

\begin{definition}\label{def:ho_digital_seq}
Let $\Scal$ be an infinite-dimensional digital sequence over $\FF_b$ with generating matrices $C_j\in \FF_b^{\NN\times \NN}$.
For $s\in \NN$, $\Scal^{(s)}$ denotes the projection of $\Scal$ onto the first $s$ coordinates, i.e., the digital sequence over $\FF_b$ with generating matrices $C_1,\ldots,C_s$.
Let $t(s)$ be a non-negative integer.
If for any $m\in \NN$ with $\alpha m\geq t(s)$ the set of the first $b^m$ points of $\Scal^{(s)}$ is an order $\alpha$ digital $(t(s),m,s)$-net over $\FF_b$, then we call $\Scal^{(s)}$ an order $\alpha$ digital $(t(s),s)$-sequence over $\FF_b$.

Moreover, if there exists a function $t\colon \NN\to \NN$ such that $\Scal^{(s)}$ is an order $\alpha$ digital $(t(s),s)$-sequence over $\FF_b$ for any $s\in \NN$, then we call $\Scal$ an infinite-dimensional order $\alpha$ digital sequence over $\FF_b$.
\end{definition}

\begin{remark}\label{rem:t-value}
It is clear from Definition~\ref{def:ho_digital_net} that any digital net can be regarded as an order $\alpha$ digital $(t,m,s)$-net with $t=\alpha m$ since then there is no linear independence condition imposed on the rows of the generating matrices.
Once a digital sequence is considered, however, it follows from Definition~\ref{def:ho_digital_seq} that the value of $t(s)$ is supposed to be independent of $m$. Since we are interested in the convergence behavior of the worst-case error as a function of $m$ (instead of a fixed $m$), we allow the situation $\alpha m <t(s)$ and think of $P_{b^1}^{(s)},P_{b^2}^{(s)},\ldots$ as a sequence of order $\alpha$ digital $(t(s),m,s)$-nets with a fixed $t(s)$.
\end{remark}

\subsubsection{Explicit construction}
Explicit constructions of infinite-dimensional order 1 digital sequences over $\FF_b$ have been given by Sobol' \cite{Sob67}, Niederreiter \cite{Nie88}, Tezuka \cite{Tez93}, Niederreiter and Xing \cite{NXbook} and others.
For instance, the Niederreiter sequence introduced in \cite{Nie88}, which was generalized thereafter in \cite{Tez93}, is constructed as follows.
Let $p_1,p_2,\ldots \in \FF_b[x]$ be distinct monic irreducible polynomials over $\FF_b$.
For each $j\in \NN$, let $e_j=\deg(p_j)$ and consider the following Laurent series expansion
\begin{align*}
 \frac{x^{e_j-z-1}}{p_j(x)^i}=\sum_{l=1}^{\infty}a^{(j)}(i,z,l)x^{-l} \in \FF_b((x^{-1})),
\end{align*}
for integers $i\geq 1$ and $0\leq z<e_j$. Then define the matrix $C_j=(c^{(j)}_{k,l})_{k,l\in \NN}$ by
\begin{align*}
 c^{(j)}_{k,l}=a^{(j)}\left(\left\lfloor \frac{k-1}{e_j} \right\rfloor +1 , (k-1)\bmod e_j,l\right) .
\end{align*}
Note that we have $c^{(j)}_{k,l}=0$ whenever $k>l$.
It is known that these matrices $C_j$ generate an infinite-dimensional order 1 digital sequence over $\FF_b$ whose function $t:\NN\to \NN$ is given by
\begin{align*}
 t(s) = \sum_{j=1}^{s}\left( e_j-1\right) ,
\end{align*}
for any $s\in \NN$.
We refer to \cite[Section~8]{DPbook} for more information on these special constructions of infinite-dimensional order 1 digital sequences.
In what follows, we introduce the digit interlacing composition due to Dick \cite{Dic07,Dic08}, which enables to explicitly construct infinite-dimensional order $\alpha$ digital sequences over $\FF_b$ for a given $\alpha\in \NN$ by using infinite-dimensional order 1 digital sequences over $\FF_b$.

\begin{definition}\label{def:interlacing}
For $\alpha\in \NN$, let $\bsx=(x_1,\ldots,x_{\alpha})\in [0,1)^{\alpha}$.
For $1\leq j\leq \alpha$, we denote the $b$-adic expansion of $x_j$ by $x_j=\xi_{1,j}/b+\xi_{2,j}/b^2+\cdots$ with $\xi_{i,j}\in \FF_b$, which is understood to be unique in the sense that infinitely many of the $\xi_{i,j}$'s are different from $b-1$.
Then we define the map $\Dcal_{\alpha}: [0,1)^{\alpha}\to [0,1)$ by
\begin{align*}
 \Dcal_{\alpha}(x_1,\ldots,x_{\alpha}) := \sum_{i=1}^{\infty}\sum_{j=1}^{\alpha}\frac{\xi_{i,j}}{b^{\alpha(i-1) +j}} .
\end{align*}
In case of an infinite-dimensional point $\bsx=(x_1,x_2,\ldots)\in [0,1)^{\NN}$, we apply $\Dcal_{\alpha}$ to every non-overlapping consecutive $\alpha$ components of $\bsx$, i.e.,
\begin{align*}
 \Dcal_{\alpha}(x_1,x_2,\ldots) := \left( \Dcal_{\alpha}(x_1,\ldots,x_{\alpha}),\Dcal_{\alpha}(x_{\alpha +1},\ldots,x_{2\alpha}),\ldots\right) \in [0,1)^{\NN}.
\end{align*}
\end{definition}

By using infinite-dimensional order $1$ digital sequences and $\Dcal_{\alpha}$, we can construct infinite-dimensional order $\alpha$ digital sequences as follows \cite[Theorems~4.11 and 4.12]{Dic08}.
\begin{lemma}\label{lem:interlacing}
Let $\Scal=\{\bsx_n\colon n\geq 0\} \subset [0,1)^{\NN}$ be an infinite-dimensional order 1 digital sequence over $\FF_b$ with $t=t_1\colon \NN\to \NN$.
For $\alpha\in \NN$, define
\begin{align*}
 \Dcal_{\alpha}(\Scal) := \{\Dcal_{\alpha}(\bsx_n)\colon n\geq 0\} \subset [0,1)^{\NN}.
\end{align*}
Then $\Dcal_{\alpha}(\Scal)$ is an infinite-dimensional order $\alpha$ digital sequence over $\FF_b$ with $t=t_{\alpha}\colon \NN\to \NN$, where $t_{\alpha}$ is given by
\begin{align*}
 t_{\alpha}(s)= \alpha t_1(\alpha s)+\frac{s\alpha(\alpha-1)}{2},
\end{align*}
for any $s\in \NN$.
\end{lemma}

\begin{remark}\label{rem:interlacing}
Let $\Scal$ be an infinite-dimensional digital sequence with generating matrices $C_j\in \FF_b^{\NN\times \NN}$.
For $i,j\in \NN$, let $\bsc_{i,j}$ denote the $i$-th row of $C_j$.
For each $j\in \NN$, define a matrix $D_j\in \FF_b^{\NN\times \NN}$, whose $i$-th row is denoted by $\bsd_{i,j}$, as
\begin{align*}
 \bsd_{\alpha (h-1)+i, j} = \bsc_{h, \alpha(j-1)+i},
\end{align*}
for $h\in \NN$ and $1 \leq i\leq \alpha$.
Then the infinite-dimensional digital sequence over $\FF_b$ with generating matrices $D_j\in \FF_b^{\NN\times \NN}$ is nothing but $\Dcal_{\alpha}(\Scal)$ defined in Lemma~\ref{lem:interlacing}.
\end{remark}

\subsubsection{Dual net and Dick metric functions}
The concept of dual net is of crucial importance in analyzing the worst-case error of a digital net.

\begin{definition}
For $m,n,s\in \NN$, let $P_{b^m}^{(s)}$ be a digital net with generating matrices $C_1,\ldots,C_s\in \FF_b^{n\times m}$.
Then the dual net of $P_{b^m}^{(s)}$, denoted by $P_{b^m}^{(s)\perp}$, is defined as
\begin{align*}
P_{b^m}^{(s)\perp} := \left\{ \bsk=(k_1,\ldots,k_s)\in \NN_0^s\colon C_1^{\top}\vec{k}_1\oplus \cdots \oplus C_s^{\top}\vec{k}_s = \bszero \in \FF_b^m \right\} ,
\end{align*}
where we set $\vec{k}=(\kappa_1,\ldots,\kappa_n)$ for $k\in \NN_0$ whose $b$-adic expansion is denoted by $k=\sum_{i=1}^{\infty}\kappa_ib^{i-1}$, where all except a finite number of $\kappa_i$ are 0.
\end{definition}

Further we recall the definition of the Dick metric function $\mu_{\alpha}$ for $\alpha\in \NN$ \cite{Dic08}.
Note that the special case where $\alpha=1$ was originally introduced in \cite{Nie86} and \cite{RT97}, and is called the NRT metric function.
\begin{definition}\label{def:Dick_metric_func}
Let $\alpha\in \NN$. For $k\in \NN$, we denote the $b$-adic expansion of $k$ by $k=\kappa_1b^{c_1-1}+\kappa_2b^{c_2-1}+\cdots+\kappa_vb^{c_v-1}$ such that $\kappa_1,\ldots,\kappa_v\in \{1,\ldots,b-1\}$ and $c_1>c_2>\cdots >c_v>0$. Then we define
\begin{align*}
 \mu_{\alpha}(k):=\sum_{i=1}^{\min(\alpha,v)}c_i ,
\end{align*}
and $\mu_{\alpha}(0):=0$. For $\bsk=(k_1,\ldots,k_s)\in \NN_0^s$, we define
\begin{align*}
 \mu_{\alpha}(\bsk):=\sum_{j=1}^{s}\mu_{\alpha}(k_j).
\end{align*}
\end{definition}
\noindent
Then the minimum Dick metric of a digital net is defined as follows.
\begin{definition}\label{def:min_Dick_metric}
For $m,n,s\in \NN$, let $P_{b^m}^{(s)}$ be a digital net over $\FF_b$ and $P_{b^m}^{(s)\perp}$ its dual net.
For $\alpha\in \NN$, the minimum Dick metric of $P_{b^m}^{(s)}$ is defined by
\begin{align*}
 \rho_{\alpha}(P_{b^m}^{(s)}) := \min_{\bsk\in P_{b^m}^{(s)\perp}\setminus \{\bszero\}}\mu_{\alpha}(\bsk).
\end{align*}
\end{definition}

\subsubsection{Some properties of higher order digital nets and sequences}
The following property of an order $\alpha$ digital $(t,m,s)$-net directly follows from the linear independence of the rows of generating matrices, see \cite[Chapter~15]{DPbook}.
\begin{lemma}\label{lem:min_Dick_metric}
For any order $\alpha$ digital $(t,m,s)$-net $P_{b^m}^{(s)}$ over $\FF_b$, we have
\begin{align*}
 \alpha m -t < \rho_{\alpha}(P_{b^m}^{(s)}) \leq \alpha m.
\end{align*}
\end{lemma}
\noindent
Moreover, the following lemma is known under the name of \emph{propagation rule}, which states that any order $\alpha$ digital net is also an order $\alpha'$ digital net as long as $1\le \alpha' <\alpha$.
The lemma again directly follows from the linear independence of the rows of the generating matrices.
We refer to \cite[Theorem~3.3]{Dic07} and \cite[Theorem~4.10]{Dic08} for the proof.
\begin{lemma}\label{lem:propagation}
For $\alpha\in \NN$, let $P_{b^m}^{(s)}$ be an order $\alpha$ digital $(t,m,s)$-net over $\FF_b$ with some integer $0\le t\le \alpha m$. Then, for any $\alpha' \in \NN$ with $1\le \alpha' <\alpha$, $P_{b^m}^{(s)}$ is also an order $\alpha'$ digital $(t_{\alpha'},m,s)$-net over $\FF_b$ with $t_{\alpha'}=\lceil t\alpha'/\alpha\rceil$.
\end{lemma}

\subsection{Sobolev spaces}\label{subsec:sobolev}
Here we introduce the function space which we consider in this paper.
First let us consider the one-dimensional case.
For $\alpha\in \NN$, the Sobolev space with smoothness $\alpha$ is given by
\begin{align*}
 H_{\alpha} & := \Big\{f \colon [0,1)\to \RR \mid \\
 & \qquad f^{(r)} \colon \text{absolutely continuous for $r=0,\ldots,\alpha-1$}, f^{(\alpha)}\in L^2([0,1))\Big\},
\end{align*}
where $f^{(r)}$ denotes the $r$-th derivative of $f$, with the inner product
\begin{align*}
 \langle f, g \rangle_{\alpha} = \sum_{r=0}^{\alpha-1}\int_{0}^{1}f^{(r)}(x)\, \rd x \int_{0}^{1}g^{(r)}(x)\, \rd x + \int_{0}^{1}f^{(\alpha)}(x)g^{(\alpha)}(x)\, \rd x,
\end{align*}
for $f,g\in H_{\alpha}$. The space $H_{\alpha}$ is also a reproducing kernel Hilbert space with the reproducing kernel
\begin{align*}
 K_{\alpha}(x,y) = \sum_{r=0}^{\alpha}\frac{B_r(x)B_r(y)}{(r!)^2}+(-1)^{\alpha+1}\frac{B_{2\alpha}(|x-y|)}{(2\alpha)!} ,
\end{align*}
for $x,y\in [0,1)$, where $B_r$ denotes the Bernoulli polynomial of degree $r$.

For the $s$-dimensional case, we consider the $s$-fold tensor product space of the one-dimensional space introduced above.
Thus the Sobolev space $H_{\alpha,s}$ which we deal with is simply given by $H_{\alpha,s}=\bigotimes_{j=1}^{s} H_{\alpha}$.
Then it follows from \cite[Section~8]{Aro50} that the reproducing kernel of the space $H_{\alpha,s}$ is the product of the reproducing kernels for the one-dimensional space $H_{\alpha}$.
Therefore, $H_{\alpha,s}$ is the reproducing kernel Hilbert space with the inner product
\begin{align*}
 \langle f, g \rangle_{\alpha,s} & = \sum_{u\subseteq \{1,\ldots,s\}}\sum_{\bsr_u\in \{0,\ldots,\alpha-1\}^{|u|}} \int_{[0,1)^{s-|u|}} \\
 & \qquad \left(\int_{[0,1)^{|u|}}f^{(\bsr_u,\bsalpha)}(\bsx)\, \rd \bsx_u\right) \left(\int_{[0,1)^{|u|}} g^{(\bsr_u,\bsalpha)}(\bsx) \, \rd \bsx_u\right) \, \rd \bsx_{\{1,\ldots,s\}\setminus u} ,
\end{align*}
for $f,g\in H_{\alpha,s}$, and the reproducing kernel
\begin{align*}
 K_{\alpha,s}(\bsx,\bsy) = \prod_{j=1}^{s}K_{\alpha}(x_j,y_j) ,
\end{align*}
for $\bsx=(x_1,\ldots,x_s),\bsy=(y_1,\ldots,y_s)\in [0,1)^s$.
In the above, we use the following notation: For $u\subseteq \{1,\ldots,s\}$ and $\bsx\in [0,1)^s$, we write $\bsx_u=(x_j)_{j\in u}$.
Moreover, for $\bsr_u=(r_j)_{j\in u}\in \{0,\ldots,\alpha-1\}^{|u|}$, $(\bsr_u,\bsalpha)$ denotes the $s$-dimensional vector whose $j$-th component equals $r_j$ if $j\in u$, and $\alpha$ otherwise.
Note that an integral and sum over the empty set is defined to be the identity operator.

\subsection{Walsh functions}
Here we recall the definition of Walsh functions.
First let us define the one-dimensional Walsh functions.
\begin{definition}
For $b\in \NN$, $b\geq 2$, let $\omega_b:=\exp(2\pi \sqrt{-1}/b)$. For $k\in \NN_0$, we denote the $b$-adic expansion of $k$ by $k=\sum_{i=1}^{\infty}\kappa_ib^{i-1}$, where all except a finite number of $\kappa_i$ are 0. The $k$-th $b$-adic Walsh function ${}_b\wal_k\colon [0,1)\to \{1,\omega_b,\ldots,\omega_b^{b-1}\}$ is defined by
\begin{align*}
{}_b\wal_k(x) := \omega_b^{\kappa_1 \xi_1+\kappa_2 \xi_2+\cdots},
\end{align*}
where we denote the unique $b$-adic expansion of $x\in [0,1)$ by $x=\sum_{i=1}^{\infty}\xi_ib^{-i}$ with $\xi_i\in \FF_b$.
\end{definition}
\noindent
The above definition can be extended to the multi-variate case as follows.
\begin{definition}
For $b\in \NN$, $b\geq 2$ and $\bsk=(k_1,\ldots,k_s)\in \NN_0^s$, the $\bsk$-th $b$-adic Walsh function ${}_b\wal_{\bsk}\colon [0,1)^s\to \{1,\omega_b,\ldots,\omega_b^{b-1}\}$ is defined by
\begin{align*}
{}_b\wal_{\bsk}(\bsx) := \prod_{j=1}^{s}{}_b\wal_{k_j}(x_j) .
\end{align*}
\end{definition}
\noindent
Since we shall always use Walsh functions in a fixed prime base $b$, we omit the subscript and simply write $\wal_k$ or $\wal_{\bsk}$.

As shown in \cite[Theorem~A.11]{DPbook}, the Walsh system $\{\wal_{\bsk}\colon \bsk\in \NN_0^s\}$ is a complete orthonormal system in $L^2([0,1)^s)$ for any $s\in \NN$.
Thus, we can define the Walsh series of $f\in L^2([0,1)^s)$ by
\begin{align*}
\sum_{\bsk\in \NN_0^s}\hat{f}(\bsk)\wal_{\bsk}(\bsx) ,
\end{align*}
where $\hat{f}(\bsk)$ denotes the $\bsk$-th Walsh coefficient of $f$ defined by
\begin{align*}
\hat{f}(\bsk) := \int_{[0,1)^s}f(\bsx)\overline{\wal_{\bsk}(\bsx)}\rd \bsx.
\end{align*}

Regarding the reproducing kernel $K_{\alpha,s}$ given in the last subsection, we have
\begin{align}
 \hat{K}_{\alpha,s}(\bsk,\bsl) & := \int_{[0,1)^{2s}}K_{\alpha,s}(\bsx,\bsy)\overline{\wal_{\bsk}(\bsx)}\wal_{\bsl}(\bsy) \rd \bsx \rd \bsy \nonumber \\
 & \: = \int_{[0,1)^{2s}}\prod_{j=1}^{s}K_{\alpha}(x_j,y_j)\overline{\wal_{k_j}(x_j)}\wal_{l_j}(y_j) \rd \bsx \rd \bsy \nonumber \\
 & \: = \prod_{j=1}^{s}\int_{[0,1)^{2}}K_{\alpha}(x_j,y_j)\overline{\wal_{k_j}(x_j)}\wal_{l_j}(y_j) \rd x_j \rd y_j = \prod_{j=1}^{s}\hat{K}_{\alpha}(k_j,l_j), \label{eq:walsh_product}
\end{align}
for any $\bsk=(k_1,\ldots,k_s),\bsl=(l_1,\ldots,l_s)\in \NN_0^s$.
As shown in \cite[Proposition~20]{BD09}, there exists a positive constant $D_{\alpha,b}$ such that 
\begin{align}\label{eq:walsh_bound}
 \left| \hat{K}_{\alpha}(k,l)\right| \leq D_{\alpha,b}b^{-\mu_{\alpha}(k)-\mu_{\alpha}(l)}
\end{align}
for any $k,l\in \NN_0$.
In order to ensure the pointwise absolute convergence of the Walsh series of $K_{\alpha,s}$, it suffices to prove that $K_{\alpha,s}$ is continuous, which is in fact trivial, and that
\begin{align}\label{eq:walsh_sum_bound}
 \sum_{\bsk,\bsl\in \NN_0^s}\left| \hat{K}_{\alpha,s}(\bsk,\bsl)\right| < \infty.
\end{align}
To address the proof of our main result first, we postpone the proof of (\ref{eq:walsh_sum_bound}) to Section~\ref{sec:walsh_sum}, where the result in Subsection~\ref{subsec:sparse} shall be used when $\alpha=1$.
Nevertheless, for any $\alpha\in \NN$, the following pointwise equality holds
\begin{align*}
 K_{\alpha,s}(\bsx,\bsy) & = \sum_{\bsk,\bsl\in \NN_0^s}\hat{K}_{\alpha,s}(\bsk,\bsl)\wal_{\bsk}(\bsx)\overline{\wal_{\bsl}(\bsy)} \\
 & = \sum_{\bsk,\bsl\in \NN_0^s}\left( \prod_{j=1}^{s}\hat{K}_{\alpha}(k_j,l_j)\right) \wal_{\bsk}(\bsx)\overline{\wal_{\bsl}(\bsy)}.
\end{align*}
\section{Proof of the main result}\label{sec:proof}

\let\temp\thetheorem
\renewcommand{\thetheorem}{\ref{thm:main}}

We repeat the main result of this paper for the reader's convenience: 
\begin{theorem}
Let $b$ be a prime, and $s,\alpha\in \NN$.
Let $\Scal$ be an infinite-dimensional order $2\alpha+1$ digital sequence over the finite field $\FF_b$, and let $P_{N}^{(s)}$ be the set of the first $N$ points of the projection of $\Scal$ onto the first $s$ coordinates for $N\in \NN$.
Then for any $m\in \NN$, we have
\begin{align*}
e^{\wor}(H_{\alpha,s}; P_{b^m}^{(s)}) \leq C_{\alpha,b,s}\frac{m^{(s-1)/2}}{b^{\alpha m}},
\end{align*}
where $C_{\alpha,b,s}$ is positive and independent of $m$.
\end{theorem}

Again we recall that the infinite-dimensional order $2\alpha+1$ digital sequence $\Scal$ can be constructed by applying the digit interlacing composition $\Dcal_{2\alpha+1}$ to the infinite-dimensional order 1 digital sequence $\Scal_1$.
Thus the point set $P_{b^m}^{(s)}$ is obtained as
\begin{align*}
\Scal_1 \to \Dcal_{2\alpha+1}(\Scal_1) = \Scal \to P_{b^m}^{(s)}.
\end{align*}
Throughout this section we write $P$ instead of $P_{b^m}^{(s)}$ for ease of notation, and denote the dual net of $P$ by $P^{\perp}$.
As stated in Remark~\ref{rem:t-value}, $P$ is an order $2\alpha+1$ digital $(t,m,s)$-net with $t$ independent of $m$.
Moreover, for $f,g\colon \NN\to \RR$, we write $f(m)\ll_{a,b} g(m)$ if there exists a positive constant $C$ which depends on some parameters ($a$ and $b$ in this case) such that $f(m)\leq C g(m)$ for all $m$.
We first give the proof of Theorem~\ref{thm:main} by using Lemma~\ref{lem:bound_Jalpha}, which is shown later.

\begin{proof}[Proof of Theorem~\ref{thm:main}]
In order to prove Theorem~\ref{thm:main}, it suffices from Definition~\ref{def:ho_digital_seq} and Remark~\ref{rem:t-value} to prove that the inequality
\begin{align}\label{eq:main_result}
e^{\wor}(H_{\alpha,s}; P_{b^m}^{(s)}) \ll_{\alpha,b,s,t}\frac{m^{(s-1)/2}}{b^{\alpha m}}
\end{align}
holds for any $m\in \NN$.

Using \cite[Theorem~15]{BD09}, (\ref{eq:walsh_product}) and (\ref{eq:walsh_bound}), we have
\begin{align*}
\left(e^{\wor}(H_{\alpha,s}; P)\right)^2 & = \sum_{\bsk,\bsl\in P^{\perp}\setminus \{\bszero\}}\hat{K}_{\alpha,s}(\bsk,\bsl) \\
& \leq \sum_{\bsk,\bsl\in P^{\perp}\setminus \{\bszero\}}\left|\hat{K}_{\alpha,s}(\bsk,\bsl)\right|  = \sum_{\bsk,\bsl\in P^{\perp}\setminus \{\bszero\}}\prod_{j=1}^{s}\left|\hat{K}_{\alpha}(k_j,l_j)\right| \\
& \ll_{\alpha, b,s} \sum_{\substack{\bsk,\bsl\in P^{\perp}\setminus \{\bszero\}\\ \hat{K}_{\alpha,s}(\bsk,\bsl)\neq 0}}\prod_{j=1}^{s}b^{-\mu_{\alpha}(k_j)-\mu_{\alpha}(l_j)} \\
& = \sum_{\substack{\bsk,\bsl\in P^{\perp}\setminus \{\bszero\}\\ \hat{K}_{\alpha,s}(\bsk,\bsl)\neq 0}}b^{-\mu_{\alpha}(\bsk)-\mu_{\alpha}(\bsl)}.
\end{align*}
Let us recall the interpolation inequality for Dick metric functions given in \cite[Lemma~3.1]{GSY1}, which states that for $\alpha,\beta\in \NN$ with $1<\alpha\leq \beta$ we have
\begin{align*}
\mu_{\alpha}(\bsk) \geq \frac{\alpha-1}{\beta-1}\mu_{\beta}(\bsk) + \frac{\beta-\alpha}{\beta-1}\mu_{1}(\bsk),
\end{align*}
for any $\bsk\in \NN_0^s$.
Since we only consider the case $\beta=2\alpha+1$ in the following, we simply write
\begin{align}\label{eq:metric_interpolation}
\mu_{\alpha}(\bsk) \geq A\mu_{2\alpha+1}(\bsk) + B\mu_{1}(\bsk),
\end{align}
where $A = (\alpha-1)/(2\alpha)$ and $B = (\alpha+1)/(2\alpha)$. It follows from the inequality (\ref{eq:metric_interpolation}) and Definition~\ref{def:min_Dick_metric} that
\begin{align}
\left(e^{\wor}(H_{\alpha,s}; P)\right)^2 & \ll_{\alpha, b,s} \sum_{\substack{\bsk,\bsl\in P^{\perp}\setminus \{\bszero\}\\ \hat{K}_{\alpha,s}(\bsk,\bsl)\neq 0}}b^{-\mu_{\alpha}(\bsk)-\mu_{\alpha}(\bsl)} \nonumber \\
& \leq \sum_{\substack{\bsk,\bsl\in P^{\perp}\setminus \{\bszero\}\\ \hat{K}_{\alpha,s}(\bsk,\bsl)\neq 0}}b^{-A(\mu_{2\alpha+1}(\bsk)+\mu_{2\alpha+1}(\bsl))-B(\mu_{1}(\bsk)+\mu_{1}(\bsl))} \nonumber \\
& \leq b^{-2A\rho_{2\alpha+1}(P)}\sum_{\substack{\bsk,\bsl\in P^{\perp}\setminus \{\bszero\}\\ \hat{K}_{\alpha,s}(\bsk,\bsl)\neq 0}}b^{-B(\mu_{1}(\bsk)+\mu_{1}(\bsl))}\nonumber \\
& = b^{-2A\rho_{2\alpha+1}(P)}\sum_{z=2\rho_1(P)}^{\infty}\sum_{\substack{\bsk,\bsl\in P^{\perp}\setminus \{\bszero\}\\ \hat{K}_{\alpha,s}(\bsk,\bsl)\neq 0, \mu_{1}(\bsk)+\mu_{1}(\bsl)=z}}b^{-Bz}, \nonumber \\
& = b^{-2A\rho_{2\alpha+1}(P)}\sum_{z=2\rho_1(P)}^{\infty}\frac{|J_{\alpha}(z)|}{b^{Bz}},\label{eq:error_bound_proof1}
\end{align}
where we define
\begin{align}\label{eq:def_Jalpha}
J_{\alpha}(z) := \{(\bsk,\bsl)\in (P^{\perp}\setminus \{\bszero\})^2\colon \hat{K}_{\alpha,s}(\bsk,\bsl)\neq 0, \mu_{1}(\bsk)+\mu_{1}(\bsl)=z\}.
\end{align}
Plugging the bound on $|J_{\alpha}(z)|$ which will be given in Lemma~\ref{lem:bound_Jalpha}, i.e., 
\begin{align*}
|J_{\alpha}(z)| \ll_{\alpha, b,s,t} (z-2\rho_1(P))^{2s\alpha+1}z^{s-1}b^{(z-2\rho_1(P))/2},
\end{align*}
into (\ref{eq:error_bound_proof1}) and using the change of variables $z=\kappa+2\rho_1(P)$, we obtain
\begin{align*}
\left(e^{\wor}(H_{\alpha,s}; P)\right)^2 & \ll_{\alpha, b,s,t} b^{-2A\rho_{2\alpha+1}(P)}\sum_{\kappa=0}^{\infty}\frac{\kappa^{2\alpha s+1}(\kappa+2\rho_1(P))^{s-1}b^{\kappa/2}}{b^{B(\kappa+2\rho_1(P))}} \\
& \ll_{s} \frac{\rho_1(P)^{s-1}}{b^{2A\rho_{2\alpha+1}(P)+2B\rho_1(P)}}\sum_{\kappa=0}^{\infty}\frac{\kappa^{s(2\alpha+1)}}{b^{\kappa/(2\alpha)}} \\
& \ll_{\alpha,b,s} \frac{m^{s-1}}{b^{2A\rho_{2\alpha+1}(P)+2B\rho_1(P)}},
\end{align*}
where the last inequality stems from the facts that the sum over $\kappa$ is finite and depends only on $s,\alpha,b$ and that $\rho_1(P)$ cannot be larger than $m$.
Finally, since it follows from Lemmas~\ref{lem:min_Dick_metric} and \ref{lem:propagation} that
\begin{align*}
2A\rho_{2\alpha+1}(P)+2B\rho_1(P) & \geq \frac{\alpha-1}{\alpha}\left((2\alpha+1)m-t\right) + \frac{\alpha+1}{\alpha}\left(m-\lceil t/(2\alpha+1)\rceil\right) \\
& = 2\alpha m -\frac{\alpha-1}{\alpha}t-\frac{\alpha+1}{\alpha}\lceil t/(2\alpha+1)\rceil ,
\end{align*}
we have
\begin{align*}
\left(e^{\wor}(H_{\alpha,s}; P)\right)^2 \ll_{\alpha, b,s,t} \frac{m^{s-1}}{b^{2\alpha m}}.
\end{align*}
Thus the result follows by taking the square root.
\end{proof}

\subsection{Sparsity of the Walsh coefficients}\label{subsec:sparse}
Here we show the sparsity of the Walsh coefficients $\hat{K}_{\alpha,s}$ in a stronger form than that given by the authors in \cite[Section~4.1]{GSY2}.
Let us consider the one-dimensional case first.
As can be seen from \cite[Section~3.1]{BD09}, the Walsh coefficient of the univariate reproducing kernel $K_{\alpha}$ is given by
\begin{align}\label{eq:sparse_walsh}
\hat{K}_{\alpha}(k,l)=\sum_{r=0}^{\alpha}\hat{b}_r(k)\overline{\hat{b}_r(l)}+(-1)^{\alpha+1}\hat{b}_{2\alpha,\rm{per}}(k,l),
\end{align}
for $k,l\in \NN_0$, where we write
\begin{align*}
\hat{b}_r(k)&:=\int_0^1\frac{B_r(x)}{r!}\overline{\wal_{k}(x)}\rd x,\\
\hat{b}_{r,\rm{per}}(k,l)&:=\int_0^1\int_0^1\frac{\tilde{B}_{r}(|x-y|)}{r!}
\overline{\wal_{k}(x)}\wal_{l}(y)\rd x\rd y,
\end{align*}
where $\tilde{B}_{r}:\RR\to \RR$ is defined by extending $B_{r}$ periodically to $\RR$.
Note that we have $B_{r}(|x-y|)=\tilde{B}_{r}(x-y)$ for even $r$ and $B_{r}(|x-y|)=(-1)^{1_{x<y}}\tilde{B}_{r}(x-y)$ for odd $r$ for any $x,y\in [0,1)$, where $1_{x<y}$ equals 1 if $x<y$ and 0 otherwise.

In order to prove that the Walsh coefficients $\hat{K}_{\alpha}$ are sparse, i.e., $\hat{K}_{\alpha}(k,l)=0$ for many choices of $k,l\in \NN_0$, we introduce the notion of type $(p,q)$.
\begin{definition}\label{def:type}
For $k,l\in \NN_0$, we denote the $b$-adic expansions of $k$ and $l$ by
$$k=\sum_{i=1}^{v}\kappa_ib^{c_i-1}\quad \text{and}\quad l=\sum_{i=1}^{w}\lambda_ib^{d_i-1},$$
respectively, where $\kappa_1,\ldots,\kappa_v,\lambda_1,\ldots,\lambda_w\in \{1,\ldots,b-1\}$, $c_1>c_2>\cdots >c_v>0$ and $d_1>d_2>\cdots >d_w>0$.
For $k=0$ ($l=0$, resp.), we assume that $v=0$ and $\kappa_0b^{c_0-1}=0$ ($w=0$ and $\lambda_0b^{d_0-1}=0$, resp.).
For $p,q\in \NN_0$, we write
$$k^{(p)}=\sum_{i=p+1}^{v}\kappa_ib^{c_i-1}\quad \text{and}\quad l^{(q)}=\sum_{i=q+1}^{w}\lambda_ib^{d_i-1},$$
where the empty sum equals 0.
Then we say that $(k,l)$ is of type $(p,q)$ if $k^{(p)}=l^{(q)}$ and $\kappa_{p}b^{c_p-1}\neq \lambda_{q}b^{d_q-1}$, where we set $\kappa_{0}b^{c_0-1}=\lambda_{0}b^{d_0-1}=0$, except the case $k=l$ where we say that $(k,l)$ is of type $(0,0)$.
\end{definition}

\begin{remark}\label{rem:type}
If $(k,l)$ is of type $(p,q)$, then it follows from the above definition that $v-p=w-q$, and that $\kappa_{v-i}=\lambda_{w-i}$ and $c_{v-i}=d_{w-i}$ for all $0\leq i< v-p$.
For any $k,l\in\NN_0$, there is a unique set of nonnegative integers $p,q$ with $p\le v$ and  $q\le w$ such that $(k,l)$ is of type $(p,q)$.
\end{remark}

Now the sparsity of the Walsh coefficients can be formulated as follows.
\begin{proposition}\label{prop:sparse_walsh}
Let $k,l\in\NN_0$ such that $(k,l)$ is of type $(p,q)$ with $p+q> 2\alpha$. Then we have $\hat{K}_{\alpha}(k,l)=0$.
\end{proposition}
\noindent In order to prove Proposition~\ref{prop:sparse_walsh}, it suffices to show that every term on the right-hand side of (\ref{eq:sparse_walsh}) is 0 whenever $(k,l)$ is of type $(p,q)$ such that $p+q> 2\alpha$, which shall be proven in Lemmas~\ref{lem:sparse_walsh_1} and \ref{lem:sparse_walsh_2} below.

\begin{lemma}\label{lem:sparse_walsh_1}
Let $k,l\in\NN_0$ such that $(k,l)$ is of type $(p,q)$ with $p+q> 2\alpha$. Then we have $\hat{b}_{r}(k)\overline{\hat{b}_{r}(l)}=0$ for all $0\leq r\leq \alpha$.
\end{lemma}
\begin{proof}
Since $p+q>2\alpha$, we must have either $p>\alpha$ or $q>\alpha$, which implies that either $v>\alpha$ or $w>\alpha$ holds, i.e., the $b$-adic expansion of either $k$ or $l$ contains more than $\alpha$ non-zero terms.
Then it is known from \cite[Lemma~3.7]{Dic08} that either the $k$-th or $l$-th Walsh coefficients of polynomials of degree less than or equal to $\alpha$ are all 0. This means either $\hat{b}_{r}(k)=0$ or $\hat{b}_{r}(l)=0$ for all $0\leq r\leq \alpha$, which completes the proof.
\end{proof}

\begin{lemma}\label{lem:sparse_walsh_2}
For $r\in \NN$, $r\geq 2$, let $k,l\in\NN_0$ such that $(k,l)$ is of type $(p,q)$ with $p+q> r$. Then we have $\hat{b}_{r,\per}(k,l)=0$.
\end{lemma}
\begin{proof}
We prove this lemma by induction on $r\geq 2$.

Let us consider the case $r=2$ first.
If $(k,l)$ is of type $(p,q)$ with $p+q> 2$, it never follows that $k=l$, $k^{(1)}=l^{(1)}$ with $k\neq l$, $k^{(1)}=l$, $k=l^{(1)}$, $k^{(2)}=l$, or $k=l^{(2)}$, since in such cases $(k,l)$ becomes of type $(p,q)$ with $p+q=$ 0, 2, 1, 1, 2 and 2, respectively.
Then the result immediately follows from \cite[Lemma~10]{Dic09}.

Now assume that the result holds true for $r-1$. That is, we assume that
\begin{align}\label{eq:sparse_walsh_proof_1}
\hat{b}_{r-1,\per}(k,l)=0 \quad \text{for $k,l\in \NN_0$ of type $(p,q)$ with $p+q>r-1$}.
\end{align}
If either $k=0$ or $l=0$ holds, the result $\hat{b}_{r,\per}(k,l)=0$ immediately follows from \cite[Lemma~11]{Dic09}.
Thus we focus on the case $k,l>0$ in the following.
Moreover, since $\hat{b}_{r,\per}(l,k)=\overline{\hat{b}_{r,\per}(k,l)}$, we can restrict ourselves to the case $p>r/2$.
As in \cite[Equation~14.18]{DPbook}, for any $k,l\in \NN$ and $r>2$ we have the identity
\begin{align*}
\hat{b}_{r,\per}(k,l) & = -\frac{1}{b^{c_1}}\Big(\frac{1}{1-\omega_b^{-\kappa_1}}\hat{b}_{r-1,\per}(k^{(1)},l)+\left( \frac{1}{2}+\frac{1}{\omega_b^{-\kappa_1}-1}\right)\hat{b}_{r-1,\per}(k,l) \\
& \qquad \qquad + \sum_{a=1}^{\infty}\sum_{\theta=1}^{b-1}\frac{1}{b^a(\omega_b^{\theta}-1)}\hat{b}_{r-1,\per}(\theta b^{a+c_1-1}+ k,l) \Big) .
\end{align*}
Thus, in order to prove $\hat{b}_{r,\per}(k,l)=0$ for $k,l\in \NN$ of type $(p,q)$ with $p+q> r$, it suffices to prove that (i) $\hat{b}_{r-1,\per}(k^{(1)},l)=0$, (ii) $\hat{b}_{r-1,\per}(k,l)=0$ and (iii) $\hat{b}_{r-1,\per}(\theta b^{a+c_1-1}+ k,l)=0$ for any $a\in \NN$ and $1\leq \theta\leq b-1$ whenever $p+q>r$ and $p>r/2$.

First, from the induction assumption (\ref{eq:sparse_walsh_proof_1}), it is trivial that $\hat{b}_{r-1,\per}(k,l)=0$ also for $k,l\in \NN$ of type $(p,q)$ with $p+q>r$. The remaining two items (i) and (iii) can be proven in the following way.

As $k^{(1)}$ can be written as
\begin{align*}
k^{(1)} = \sum_{i=2}^{p}\kappa_ib^{c_i-1}+\sum_{i=q+1}^{w}\lambda_{i}b^{d_i-1},
\end{align*}
where $\kappa_pb^{c_p-1}\neq \lambda_qb^{d_q-1}$, $(k^{(1)},l)$ is of type $(p-1,q)$, where $p-1+q>r-1$.
Thus, again from the induction assumption (\ref{eq:sparse_walsh_proof_1}), it follows that $\hat{b}_{r-1,\per}(k^{(1)},l)=0$, which completes the proof of the item (i).

Similarly, as $\theta b^{a+c_1-1}+ k$ can be written as
\begin{align*}
\theta b^{a+c_1-1}+ k = \theta b^{a+c_1-1}+ \sum_{i=1}^{p}\kappa_ib^{c_i-1}+\sum_{i=q+1}^{w}\lambda_{i}b^{d_i-1},
\end{align*}
where $\kappa_pb^{c_p-1}\neq \lambda_qb^{d_q-1}$, for any $a\in \NN$ and $1\leq \theta\leq b-1$, $(\theta b^{a+c_1-1}+ k,l)$ is of type $(p+1,q)$, where $p+1+q>r+1$.
Again from the induction assumption (\ref{eq:sparse_walsh_proof_1}), it follows that $\hat{b}_{r-1,\per}(\theta b^{c+a_1-1}+ k,l)=0$, which completes the proof of the item (iii).
\end{proof}

Let us move on to the high-dimensional case. As a corollary of Proposition~\ref{prop:sparse_walsh}, the sparsity of the Walsh coefficients $\hat{K}_{\alpha,s}$ can be formulated as follows.
\begin{corollary}\label{cor:sparse_walsh}
Let $\bsk=(k_1,\ldots,k_s),\bsl=(l_1,\ldots,l_s)\in\NN_0$ such that $(k_j,l_j)$ is of type $(p_j,q_j)$ with $p_j+q_j> 2\alpha$ for at least one index $j\in \{1,\ldots,s\}$.
Then we have $\hat{K}_{\alpha,s}(\bsk,\bsl)=0$.
\end{corollary}
\begin{proof}
For an index $j\in \{1,\ldots,s\}$ such that $(k_j,l_j)$ is of type $(p_j,q_j)$ with $p_j+q_j> 2\alpha$, it follows from Proposition~\ref{prop:sparse_walsh} that $\hat{K}_{\alpha}(k_j,l_j)=0$.
The result then follows from (\ref{eq:walsh_product}).
\end{proof}

\subsection{A bound on $|J_{\alpha}(z)|$}\label{subsec:cardinality}
Here we give a bound on the cardinality of $J_{\alpha}(z)$ defined in (\ref{eq:def_Jalpha}) by using the result on the sparsity of the Walsh coefficients $\hat{K}_{\alpha,s}$ and the property of higher order digital nets.
For this purpose, we introduce the following two sets
\begin{align*}
R_{\alpha}(\bsk, z_2) & := \{\bsl\in P^{\perp} \setminus \{\bszero\} \colon \mu_1(\bsl)=z_2, \hat{K}_{\alpha,s}(\bsk,\bsl)\neq 0\}, \\
J_{\alpha}(z_1,z_2) & := \{(\bsk ,\bsl)\in (P^{\perp} \setminus \{\bszero\})^2 \colon \mu_1(\bsk)=z_1, \mu_1(\bsl)=z_2, \hat{K}_{\alpha,s}(\bsk,\bsl)\neq 0\},
\end{align*}
for $z_1,z_2\in \NN$ and $\bsk\in \NN_0^s$.
In what follows, $\binom{i}{j}$ denotes the binomial coefficient, where we set $\binom{i}{j}=0$ if $j>i$.

We first give a bound on $|R_{\alpha}(\bsk, z_2)|$.
It is obvious from Definition~\ref{def:min_Dick_metric} that $\mu_1(\bsl)\geq \rho_1(P)$ for $\bsl\in P^{\perp} \setminus \{\bszero\}$.
Thus $|R_{\alpha}(\bsk, z_2)|=0$ for $z_2< \rho_1(P)$.
For $z_2\geq \rho_1(P)$ the following holds true.
\begin{lemma}\label{lem:bound_Ralpha}
Let $P$ be an order $2\alpha+1$ digital $(t,m,s)$-net over $\FF_b$. For $z_2\geq \rho_1(P)$ and $\bsk\in \NN_0^s$, we have
\begin{align*}
|R_{\alpha}(\bsk, z_2)| \ll_{\alpha,b,s} \prod_{i=1}^{2\alpha}\binom{(i+1)z_2-\rho_{i+1}(P)+s}{s} .
\end{align*}
\end{lemma}

\begin{proof}
For $\bsp=(p_1,\ldots,p_s),\bsq=(q_1,\ldots,q_s)\in \NN_0^s$, define the set
\begin{align*}
R_{\bsp,\bsq}(\bsk, z_2) & := \{\bsl\in P^{\perp} \setminus \{\bszero\} \colon \mu_1(\bsl)=z_2, \text{$(k_j,l_j)$ is of type $(p_j,q_j)$}\}.
\end{align*}
It follows from Corollary~\ref{cor:sparse_walsh} that in order to have $\hat{K}_{\alpha,s}(\bsk,\bsl)\neq 0$, $(k_j,l_j)$ must be of type $(p_j,q_j)$ with $p_j+q_j\leq 2\alpha$ for all $j=1,\ldots,s$.
Thus we have
\begin{align}\label{eq:Ralpha_inclusion}
R_{\alpha}(\bsk, z_2) \subset \bigcup_{\substack{\bsp,\bsq\in \NN_0^s\\ p_j+q_j\leq 2\alpha, \forall j}}R_{\bsp,\bsq}(\bsk, z_2) .
\end{align}
Note that each $R_{\bsp,\bsq}(\bsk, z_2)$ is a subset of $P^{\perp}$ and that the above union is disjoint.
Further define the set
\begin{align*}
S_{2\alpha}(z_2) & := \Big\{(\delta_{i,j},\zeta_{i,j})_{1\leq i\leq 2\alpha, 1\leq j\leq s}\in (\NN_0 \times \{1,\ldots,b-1\})^{2s\alpha} \\
& \qquad \colon \sum_{j=1}^{s}\delta_{i,j}\leq (i+1)z_2-\rho_{i+1}(P),\quad i=1,\ldots,2\alpha \Big\} ,
\end{align*}
and the mapping $\phi_{2\alpha}$ from $\NN_0^s$ to $(\NN_0 \times \{1,\ldots,b-1\})^{2s\alpha}$ by
\begin{align*}
(l_j)_{1\leq j\leq s} \mapsto (d_{i,j}-d_{i+1,j},\lambda_{i,j})_{1\leq i\leq 2\alpha, 1\leq j\leq s},
\end{align*}
where we denote the $b$-adic expansion of $l_j$ by $l_j=\sum_{i=1}^{w_j}\lambda_{i,j}b^{d_{i,j}-1}$. For $i>w_{j}$, we set $(d_{i,j},\lambda_{i,j})=(0,1)$.

First we show that the image of the restriction of $\phi_{2\alpha}$ to $R_{\bsp,\bsq}(\bsk, z_2)$ is included in $S_{2\alpha}(z_2)$ for any $\bsp,\bsq\in \NN_0^s$ with $p_j+q_j\leq 2\alpha$.
From Definition~\ref{def:min_Dick_metric} we have
\begin{align*}
\mu_{i}(\bsl) = \sum_{j=1}^{s}\sum_{l=1}^{i}d_{l,j} \geq \rho_{i}(P),
\end{align*}
for any $i \in \NN$ and $\bsl\in P^{\perp}\setminus \{\bszero\}$. Thus it follows that again for any $i\in \NN$
\begin{align*}
\sum_{j=1}^{s}(d_{i,j}-d_{i+1,j}) & = \sum_{j=1}^{s}d_{i,j} -\left(\sum_{j=1}^{s}\sum_{l=1}^{i+1}d_{l,j}-\sum_{j=1}^{s}\sum_{l=1}^{i}d_{l,j}\right) \\
& \leq \sum_{j=1}^{s}d_{i,j}-\rho_{i+1}(P)+\sum_{j=1}^{s}\sum_{l=1}^{i}d_{l,j} \\
& \leq \sum_{j=1}^{s}d_{1,j}-\rho_{i+1}(P)+\sum_{j=1}^{s}i d_{1,j} \\
& = (i+1)\mu_{1}(\bsl)-\rho_{i+1}(P) = (i+1)z_2-\rho_{i+1}(P) .
\end{align*}
Thus it implies that $(d_{i,j}-d_{i+1,j},\lambda_{i,j})_{1\leq i\leq 2\alpha, 1\leq j\leq s}\in S_{2\alpha}(z_2)$.

Next we show that the restricted map $\phi_{2\alpha}|_{R_{\bsp,\bsq}(\bsk, z_2)}\colon R_{\bsp,\bsq}(\bsk, z_2)\to S_{2\alpha}(z_2)$ is injective.
We denote the $b$-adic expansion of $k_j$ by $k_j=\sum_{i=1}^{v_j}\kappa_{i,j}b^{c_{i,j}-1}$ for $\bsk = (k_1, \dots, k_s)\in \NN_0^s$.
Let $\bsl \in R_{\bsp,\bsq}(\bsk, z_2)$ and $\phi_{2\alpha}(\bsl)=(\delta_{i,j},\zeta_{i,j})_{1\le i\le 2\alpha,1\le j\le s}$.
It suffices to show that $\bsl$ is determined by $(\delta_{i,j},\zeta_{i,j})_{1\le i\le 2\alpha, 1\le j\le s}$, $\bsk$, $\bsp$ and $\bsq$.
Since $(k_j,l_j)$ is of type $(p_j,q_j)$, we have
\[
\sum_{i=q_j+1}^{w_j}\lambda_{i,j}b^{d_{i,j}-1} = \sum_{i=p_j+1}^{v_j}\kappa_{i,j}b^{c_{i,j}-1},
\]
where the empty sum equals 0, and $d_{q_j + 1,j} = c_{p_j + 1,j}$.
Note that if $q_j = w_j$, which is equivalent to $p_j = v_j$, then we set $d_{w_j + 1,j} = c_{v_j + 1,j} = 0$.
Further, by the definition of $\phi_{2\alpha}$ and the fact $q_j \leq 2\alpha$, for all $1 \leq i \leq q_j$ we have
\[\lambda_{i,j}=\zeta_{i,j} 
\quad \text{and}
\quad d_{i,j} = \sum_{h=i}^{q_j} \delta_{h,j} + d_{q_j + 1,j} = \sum_{h=i}^{q_j} \delta_{h,j} + c_{p_j + 1,j}.
\]
Thus $l_j$ can be determined by $(\delta_{i,j},\zeta_{i,j})_{1\le i\le 2\alpha}$, $k_j$, $p_j$ and $q_j$ as
\begin{align*}
l_j = \sum_{i=1}^{w_j}\lambda_{i,j}b^{d_{i,j}-1} = \sum_{i=1}^{q_j}\zeta_{i,j}b^{c_{p_j+1,j}+\sum_{h=i}^{q_j}\delta_{h,j}-1}+\sum_{i=p_j+1}^{v_j}\kappa_{i,j}b^{c_{i,j}-1},
\end{align*}
where the empty sum equals 0.
This implies the injectivity of $\phi_{2\alpha}|_{R_{\bsp,\bsq}(\bsk, z_2)}$.

By using (\ref{eq:Ralpha_inclusion}) and the result that $\phi_{2\alpha}|_{R_{\bsp,\bsq}(\bsk, z_2)}\colon R_{\bsp,\bsq}(\bsk, z_2)\to S_{2\alpha}(z_2)$ is injective, we obtain
\begin{align*}
|R_{\alpha}(\bsk, z_2)| & \leq \sum_{\substack{\bsp,\bsq\in \NN_0^s\\ p_j+q_j\leq 2\alpha, \forall j}}|R_{\bsp,\bsq}(\bsk, z_2)| \leq \sum_{\substack{\bsp,\bsq\in \NN_0^s\\ p_j+q_j\leq 2\alpha, \forall j}}|S_{2\alpha}(z_2)| \\
& = ((\alpha+1)(2\alpha+1))^s |S_{2\alpha}(z_2)| \ll_{\alpha,s}|S_{2\alpha}(z_2)|,
\end{align*}
where the number $(\alpha+1)(2\alpha+1)$ represents the total number of possible combinations of $p,q\in \NN_0$ such that $p+q\leq 2\alpha$.
Finally the cardinality of $S_{2\alpha}(z_2)$ is given by
\begin{align*}
|S_{2\alpha}(z_2)| & = (b-1)^{2\alpha s}\prod_{i=1}^{2\alpha}\left| \left\{(\delta_{i,j})_{1\leq j\leq s}\in \NN_0^{s} \colon \sum_{j=1}^{s}\delta_{i,j}\leq (i+1)z_2-\rho_{i+1}(P) \right\}\right| \\
& = (b-1)^{2\alpha s}\prod_{i=1}^{2\alpha}\sum_{r_i=0}^{(i+1)z_2-\rho_{i+1}(P)}\left| \left\{(\delta_{i,j})_{1\leq j\leq s}\in \NN_0^{s} \colon \sum_{j=1}^{s}\delta_{i,j}=r_i \right\}\right| \\
& = (b-1)^{2\alpha s}\prod_{i=1}^{2\alpha}\sum_{r_i=0}^{(i+1)z_2-\rho_{i+1}(P)}\binom{r_i+s-1}{s-1} \\
& = (b-1)^{2\alpha s}\prod_{i=1}^{2\alpha}\binom{(i+1)z_2-\rho_{i+1}(P)+s}{s},
\end{align*}
which completes the proof.
\end{proof}

We move on to a bound on $|J_{\alpha}(z_1,z_2)|$.
It is again obvious from Definition~\ref{def:min_Dick_metric} that $\mu_1(\bsk),\mu_1(\bsl)\geq \rho_1(P)$ for $\bsk,\bsl\in P^{\perp} \setminus \{\bszero\}$.
Thus $|J_{\alpha}(z_1,z_2)|=0$ when either $z_1< \rho_1(P)$ or $z_2< \rho_1(P)$.
For $z_1,z_2\geq \rho_1(P)$ the following holds true.
\begin{lemma}\label{lem:bound_Jalpha2}
Let $P$ be an order $2\alpha+1$ digital $(t,m,s)$-net over $\FF_b$. For $z_1,z_2\geq \rho_1(P)$, we have
\begin{align*}
|J_{\alpha}(z_1,z_2)| \ll_{\alpha,b,s} b^{z_1-\rho_1(P)+1}\binom{z_1+s-1}{s-1}\prod_{i=1}^{2\alpha}\binom{(i+1)z_2-\rho_{i+1}(P)+s}{s} .
\end{align*}
\end{lemma}

\begin{proof}
From the definitions of $R_{\alpha}(\bsk, z_2)$ and $J_{\alpha}(z_1,z_2)$, we have
\begin{align*}
|J_{\alpha}(z_1,z_2)| & \leq |\{\bsk\in P^{\perp} \setminus \{\bszero\} \colon \mu_1(\bsk)=z_1\}| \\
& \qquad \times \max_{\substack{\bsk\in P^{\perp}\setminus \{\bszero\}\\ \mu_1(\bsk)=z_1}}|\{\bsl\in P^{\perp} \setminus \{\bszero\} \colon \mu_1(\bsl)=z_2, \hat{K}_{\alpha,s}(\bsk,\bsl)\neq 0\}| \\
& = |I(z_1)|\max_{\substack{\bsk\in P^{\perp}\setminus \{\bszero\}\\ \mu_1(\bsk)=z_1}}|R_{\alpha}(\bsk, z_2)| ,
\end{align*}
where we write
\begin{align*}
I(z_1) := \{\bsk\in P^{\perp} \setminus \{\bszero\} \colon \mu_1(\bsk)=z_1\} .
\end{align*}
By adapting \cite[Lemma~2.2]{Skr06}, we have
\begin{align*}
|I(z_1)| & = |\{\bsk\in P^{\perp} \setminus \{\bszero\} \colon \mu_1(\bsk)=z_1\}| \\
& = \sum_{\substack{d_1,\ldots,d_s\in \NN_0\\ d_1+\cdots+d_s=z_1}}|\{\bsk\in P^{\perp} \setminus \{\bszero\} \colon \mu_1(k_j)=d_j\}| \\
& \leq \sum_{\substack{d_1,\ldots,d_s\in \NN_0\\ d_1+\cdots+d_s=z_1}}b^{d_1+\cdots+d_s-\rho_1(P)+1} = b^{z_1-\rho_1(P)+1}\binom{z_1+s-1}{s-1}.
\end{align*}
Moreover, the bound on $|R_{\alpha}(\bsk, z_2)|$ obtained in Lemma~\ref{lem:bound_Ralpha} is independent of the choice of $\bsk$.
Hence the result follows.
\end{proof}

Finally we give a bound on $|J_{\alpha}(z)|$. Since $\mu_1(\bsk),\mu_1(\bsl)\geq \rho_1(P)$ for $\bsk,\bsl\in P^{\perp} \setminus \{\bszero\}$, $|J_{\alpha}(z)|=0$ for $z< 2\rho_1(P)$.
For $z\geq 2\rho_1(P)$ we have the following result.
\begin{lemma}\label{lem:bound_Jalpha}
Let $P$ be an order $2\alpha+1$ digital $(t,m,s)$-net over $\FF_b$. For $z\geq 2\rho_1(P)$, we have
\begin{align*}
|J_{\alpha}(z)| \ll_{\alpha, b,s,t} (z-2\rho_1(P))^{2s\alpha+1}z^{s-1}b^{(z-2\rho_1(P))/2}.
\end{align*}
\end{lemma}

\begin{proof}
For $z\geq 2\rho_1(P)$, we have
\begin{align*}
J_{\alpha}(z)= \bigcup_{\substack{z_1,z_2\geq \rho_1(P)\\ z_1+z_2=z}}J_{\alpha}(z_1,z_2).
\end{align*}
Since the sets on the right-hand side above are pairwise disjoint, we have
\begin{align*}
|J_{\alpha}(z)| & = \sum_{\substack{z_1,z_2\geq \rho_1(P)\\ z_1+z_2=z}}|J_{\alpha}(z_1,z_2)| \leq 2\sum_{\substack{\rho_1(P)\leq z_1\leq \lfloor z/2\rfloor \\ z_2=z-z_1}}|J_{\alpha}(z_1,z_2)|
\end{align*}
where the last inequality stems from the symmetry of the Walsh coefficients, i.e., $\hat{K}_{\alpha,s}(\bsl,\bsk)=\overline{\hat{K}_{\alpha,s}(\bsk,\bsl)}$ for any $\bsk,\bsl\in \NN_0^s$.
By using the bound on $|J_{\alpha}(z_1,z_2)|$ given in Lemma~\ref{lem:bound_Jalpha2}, we have
\begin{align*}
|J_{\alpha}(z)| \ll_{\alpha,b,s} \sum_{\substack{\rho_1(P)\leq z_1\leq \lfloor z/2\rfloor \\ z_2=z-z_1}}b^{z_1-\rho_1(P)+1}\binom{z_1+s-1}{s-1}\prod_{i=1}^{2\alpha}\binom{(i+1)z_2-\rho_{i+1}(P)+s}{s}.
\end{align*}
It follows from Lemmas~\ref{lem:min_Dick_metric} and \ref{lem:propagation} that for $1\leq i\leq 2\alpha$
\begin{align*}
\binom{(i+1)z_2-\rho_{i+1}(P)+s}{s}& \leq \binom{(i+1)(z_2-m)+\lceil t(i+1)/(2\alpha+1) \rceil+s-1}{s} \\
& \ll_{\alpha,b,s,t} (z_2-\rho_1(P))^{s}.
\end{align*}
Also we have
\begin{align*}
\binom{z_1+s-1}{s-1} \ll_{s} z_1^{s-1} .
\end{align*}
Hence
\begin{align*}
|J_{\alpha}(z)| & \ll_{\alpha,b,s,t} \sum_{\substack{\rho_1(P)\leq z_1\leq \lfloor z/2\rfloor \\ z_2=z-z_1}}b^{z_1-\rho_1(P)+1}z_1^{s-1}(z_2-\rho_1(P))^{2\alpha s} \\
& \ll_{b,s} (z-2\rho_1(P)) b^{(z-2\rho_1(P))/2} z^{s-1}(z-2\rho_1(P))^{2\alpha s} .
\end{align*}
Thus the result follows.
\end{proof}

\section{Summability of the Walsh coefficients}\label{sec:walsh_sum}
In this section we prove:
\begin{lemma}\label{lem:walsh_sum}
For $\alpha, s\in \NN$, 
\begin{align*}
 \sum_{\bsk,\bsl\in \NN_0^s}\left| \hat{K}_{\alpha,s}(\bsk,\bsl)\right| < \infty.
\end{align*}
\end{lemma}
\noindent We would recall that this summability is needed to ensure the pointwise absolute convergence of the Walsh series of $K_{\alpha,s}$.
\begin{proof}
For $\alpha\geq 2$, by using (\ref{eq:walsh_product}) and (\ref{eq:walsh_bound}), we have
\begin{align*}
 \sum_{\bsk,\bsl\in \NN_0^s}\left| \hat{K}_{\alpha,s}(\bsk,\bsl)\right| & = \sum_{\bsk,\bsl\in \NN_0^s}\prod_{j=1}^{s}\left| \hat{K}_{\alpha,s}(k_j,l_j)\right| = \left( \sum_{k,l\in \NN_0}\left|\hat{K}_{\alpha}(k,l)\right|\right)^s \\
 & \leq  D_{\alpha,b}^s\left( \sum_{k,l\in \NN_0}b^{-\mu_{\alpha}(k)-\mu_{\alpha}(l)}\right)^s = D_{\alpha,b}^s\left( \sum_{k\in \NN_0}b^{-\mu_{\alpha}(k)}\right)^{2s}.
\end{align*}
As shown in \cite[Lemma~2.10]{Dic07}, the last sum is bounded above by $1+\alpha+b^{-2}$. Thus the result follows.

Let $\alpha=1$. Since the sum $\sum_{k\in \NN_0}b^{-\mu_1(k)}$ is no longer finite, we need to exploit the sparsity of the Walsh coefficients obtained in Proposition~\ref{prop:sparse_walsh}. Since
\begin{align*}
 \sum_{\bsk,\bsl\in \NN_0^s}\left| \hat{K}_{1,s}(\bsk,\bsl)\right| & = \left( \sum_{k,l\in \NN_0}\left|\hat{K}_{1}(k,l)\right|\right)^s \leq D_{1,b}^s\left( \sum_{\substack{k,l\in \NN_0 \\ \hat{K}_{1}(k,l)\neq 0}} b^{-\mu_1(k)-\mu_1(l)}\right)^s,
\end{align*}
it suffices to prove that the last sum is finite. It follows from Proposition~\ref{prop:sparse_walsh} that $\hat{K}_{1}(k,l)= 0$ whenever $(k,l)$ is of type $(p,q)$ with $p+q> 2$, so that
\begin{align}
 \sum_{\substack{k,l\in \NN_0 \\ \hat{K}_{1}(k,l)\neq 0}} b^{-\mu_1(k)-\mu_1(l)} & = \sum_{\substack{p,q\in \NN_0\\ p+q\leq 2}}\sum_{\substack{k,l\in \NN_0 \\ \text{$(k,l)$: type $(p,q)$}}} b^{-\mu_1(k)-\mu_1(l)} \nonumber \\
 & = \sum_{\substack{p,q\in \NN_0\\ p+q\leq 2}}\sum_{z_1,z_2\in \NN_0}b^{-z_1-z_2}|\Jcal_{p,q}(z_1,z_2)|,\label{eq:walsh_sum_smoothness1}
\end{align}
where we define the set
\begin{align*}
 \Jcal_{p,q}(z_1,z_2) := \{(k,l)\in \NN_0^s\colon \mu_1(k)=z_1,\mu_1(l)=z_2, \text{$(k,l)$ is of type $(p,q)$}\}.
\end{align*}
In what follows, we estimate the cardinality of $\Jcal_{p,q}(z_1,z_2)$ and the inner sum of (\ref{eq:walsh_sum_smoothness1}) for possible choices of $p$ and $q$. Due to the symmetry of $k$ and $l$, we only consider the case $p\geq q$.
\begin{enumerate}
\item $p=q=0$: Since it follows that $k=l$, $\Jcal_{p,q}(z_1,z_2)$ is empty if $z_1\neq z_2$. Otherwise, the cardinality of $\Jcal_{p,q}(z_1,z_2)$ equals the number of possible choices for $k$ such that $\mu_1(k)=z_1$. Thus we have
\begin{align*}
 |\Jcal_{p,q}(z_1,z_2)| = \begin{cases}
 0 & \text{if $z_1\neq z_2$,} \\
 \lceil b^{z_1-1}(b-1)\rceil & \text{otherwise,}
 \end{cases}
\end{align*}
and 
\begin{align*}
 \sum_{z_1,z_2\in \NN_0}b^{-z_1-z_2}|\Jcal_{p,q}(z_1,z_2)| = \sum_{z_1\in \NN_0}b^{-2z_1}\lceil b^{z_1-1}(b-1)\rceil = 1+\frac{1}{b}.
\end{align*}
\item $p=1,q=0$: Since $k$ must be of the form $\kappa_1 b^{z_1-1}+l$, $\Jcal_{p,q}(z_1,z_2)$ is empty if $z_1\leq z_2$. Otherwise, the cardinality of $\Jcal_{p,q}(z_1,z_2)$ equals the number of possible choices for $l$ such that $\mu_1(l)=z_2$ times the number of possible choices for $\kappa_1\in \{1,\ldots,b-1\}$. Thus we have
\begin{align*}
 |\Jcal_{p,q}(z_1,z_2)| = \begin{cases}
 0 & \text{if $z_1\leq z_2$,} \\
 \lceil b^{z_2-1}(b-1)\rceil (b-1)& \text{otherwise,}
 \end{cases}
\end{align*}
and 
\begin{align*}
 \sum_{z_1,z_2\in \NN_0}b^{-z_1-z_2}|\Jcal_{p,q}(z_1,z_2)| & = \sum_{z_1>z_2}b^{-z_1-z_2}\lceil b^{z_2-1}(b-1)\rceil (b-1) \\
 & \leq (b-1)\sum_{z_1>z_2}b^{-z_1-z_2}b^{z_2} \\
 & = (b-1)\sum_{z_1=1}^{\infty}b^{-z_1}\sum_{z_2=0}^{z_1-1}1 \\
 & = (b-1)\sum_{z_1=1}^{\infty}b^{-z_1}z_1 < \infty.
\end{align*}
\item $p=2,q=0$: Since $k$ must be of the form $\kappa_1 b^{z_1-1}+\kappa_2 b^{c_2-1}+l$, $\Jcal_{p,q}(z_1,z_2)$ is empty if $z_1\leq z_2+1$. Otherwise, the cardinality of $\Jcal_{p,q}(z_1,z_2)$ equals the number of possible choices for $l$ such that $\mu_1(l)=z_2$ times the number of possible choices for $\kappa_1,\kappa_2\in \{1,\ldots,b-1\}$ and $c_2\in \{z_2+1,\ldots,z_1-1\}$. Thus we have
\begin{align*}
 |\Jcal_{p,q}(z_1,z_2)| = \begin{cases}
 0 & \text{if $z_1\leq z_2+1$,} \\
 \lceil b^{z_2-1}(b-1)\rceil (b-1)^2(z_1-z_2-1)& \text{otherwise,}
 \end{cases}
\end{align*}
and
\begin{align*}
 & \sum_{z_1,z_2\in \NN_0}b^{-z_1-z_2}|\Jcal_{p,q}(z_1,z_2)| \\
 & = \sum_{z_1>z_2+1}b^{-z_1-z_2}\lceil b^{z_2-1}(b-1)\rceil (b-1)^2(z_1-z_2-1) \\
 & \leq (b-1)^2\sum_{z_1>z_2+1}b^{-z_1-z_2}b^{z_2} (z_1-z_2-1) \\
 & \leq (b-1)^2\sum_{z_1=2}^{\infty}b^{-z_1}\sum_{z_2=0}^{z_1-2} (z_1-z_2-1) \\
 & = \frac{(b-1)^2}{2}\sum_{z_1=2}^{\infty}b^{-z_1}(z_1-1)z_1 < \infty.
\end{align*}
\item $p=q=1$: Since $k$ and $l$ are given of the forms $k=\kappa_1b^{z_1-1}+k'$ and $l=\lambda_1b^{z_2-1}+k'$, respectively, with some $0\leq k'<b^{\min(z_1,z_2)-1}$ and $\kappa_1b^{z_1-1} \neq \lambda_1b^{z_2-1}$, the cardinality of $\Jcal_{p,q}(z_1,z_2)$ equals the number of possible choices for $k'$ times the number of possible choices for $\kappa_1,\lambda_1\in \{1,\ldots,b-1\}$. Note that if $z_1=z_2$, we have $\kappa_1 \neq \lambda_1$. Thus we have
\begin{align*}
 |\Jcal_{p,q}(z_1,z_2)| & = \begin{cases}
 \lceil b^{\min(z_1,z_2)-1}\rceil (b-1)(b-2) & \text{if $z_1=z_2$,} \\
 \lceil b^{\min(z_1,z_2)-1}\rceil (b-1)^2 & \text{otherwise,}
 \end{cases}\\
 & \leq b^{\min(z_1,z_2)}(b-1),
\end{align*}
and
\begin{align*}
 \sum_{z_1,z_2\in \NN_0}b^{-z_1-z_2}|\Jcal_{p,q}(z_1,z_2)| & = (b-1)\sum_{z_1,z_2\in \NN_0}b^{-z_1-z_2}b^{\min(z_1,z_2)} \\
 & \leq 2(b-1)\sum_{z_1\geq z_2\geq 0}b^{-z_1-z_2}b^{z_2} \\
 & = 2(b-1)\sum_{z_1= 0}^{\infty}b^{-z_1}\sum_{z_2=0}^{z_1}1 \\
 & = 2(b-1)\sum_{z_1= 0}^{\infty}b^{-z_1}(z_1+1)  < \infty.
\end{align*}
\end{enumerate}
In this way, since we only have six choices for $p$ and $q$, and for each choice the inner sum of (\ref{eq:walsh_sum_smoothness1}) is shown to be finite, (\ref{eq:walsh_sum_smoothness1}) itself is also finite. Thus the result follows.
\end{proof}
\section*{Acknowledgement}
The work of T.~G. is supported by JSPS Grant-in-Aid for Young Scientists No.15K20964.
The work of K.~S. and T.~Y. is supported by Australian Research Council's Discovery Projects funding scheme (project number DP150101770).
The authors are grateful to the reviewers for many helpful comments, particularly, for asking for the endpoint case $\alpha=1$.


\end{document}